\documentclass[11pt, leqno]{article}
\usepackage{amsmath,amssymb,amsbsy,amsfonts,amsthm,latexsym,
            amsopn,amstext,amsxtra,euscript,amscd,amsthm}

\newtheorem{lem}{Lemma}
\newtheorem{lemma}[lem]{Lemma}

\newtheorem{thm}{Theorem}
\newtheorem{theorem}[thm]{Theorem}

\def\\{\cr}
\def\({\left(}
\def\){\right)}
\def\[{\left[}
\def\]{\right]}
\def\<{\langle}
\def\>{\rangle}

\begin{document}

\title{On members of Lucas sequences which are products of factorials}

\author{
{\sc Shanta~Laishram}\\
{Stat-Math Unit, Indian Statistical Institute}\\
{7, S. J. S. Sansanwal Marg, New Delhi, 110016, India}\\
{shanta@isid.ac.in}
\and
{\sc Florian~Luca}\\
{School of Mathematics, University of the Witwatersrand}\\
{Private Bag 3, Wits 2050, South Africa}\\
{Department of Mathematics, University of Ostrava}\\
{30 Dubna 22, 701 03}\\
{Ostrava 1, Czech Republic}\\
{florian.luca@wits.ac.za}
\and
{\sc Mark~Sias}\\
{Department of Pure and Applied Mathematics}\\
{University of Johannesburg}\\
{PO Box 524, Auckland Park 2006, South Africa.}\\
{msias@uj.ac.za}}

\date{\today}
\pagenumbering{arabic}

\maketitle

\begin{abstract}
Here, we show that if $\{U_n\}_{n\ge 0}$ is a Lucas sequence, then the largest $n$ such that $|U_n|=m_1!m_2!\cdots m_k!$ with $1<m_1\le m_2\le \cdots\le m_k$ satisfies $n<3\times 10^5$. We also give better bounds in case the roots of the Lucas sequence are real. 
\end{abstract}

\section{Introduction}

Let $r,~s$ be coprime nonzero integers with $r^2+4s\ne 0$. Let $\alpha,~\beta$ be the roots of the quadratic equation $x^2-rx-s=0$. We assume further that $\alpha/\beta$ is not a root of $1$.
The Lucas sequences $\{U_n\}_{n\ge 0}$ and $\{V_n\}_{n\ge 0}$ of parameters $(r,s)$ are given by
$$
U_n=\frac{\alpha^n-\beta^n}{\alpha-\beta}\qquad {\text{\rm and}}\qquad V_n=\alpha^n+\beta^n\qquad {\text{\rm for~all}}\qquad n\ge 0.
$$
Alternatively, they can be defined recursively as $U_0=0,~U_1=1,~V_0=2,~V_1=r$ and both recurrences
$$
U_{n+2}=rU_{n+1}+sU_n\quad {\text{\rm and}}\quad V_{n+2}=rV_{n+1}+sV_n\qquad  {\text{\rm hold for all}}\quad n\ge 0.
$$ 
Let 
$$
{\mathcal PF}:=\{\pm \prod_{j=1}^k m_j!: 1<m_1\le m_2\le\cdots \le m_k~{\text{\rm and}}~k\ge 1\}
$$
be the set of integers which are product of factorials $>1$ (an empty product is interpreted as $1$). 
In \cite{Lu}, it was shown that if $t\ge 1$ is any fixed integer, then the Diophantine equation
\begin{equation}
\label{eq:1}
\prod_{i=1}^t U_{n_i}\in {\mathcal PF}
\end{equation}
has only finitely many positive integer solutions $n_1\le n_2\le \cdots\le n_t$ and they are all effectively computable. When $(r,s)=(1,1)$ then $U_n=F_n$ is the $n$th Fibonacci number. For this particular case, it was shown in \cite{LuSt} that the largest solution of equation \eqref{eq:1} with the additional restriction that $1\le n_1<n_2<\cdots<n_t$ is 
$$
F_1F_2F_3F_4F_5F_6F_8F_{10}F_{12}=11!
$$
Similar results can be proved when in \eqref{eq:1} all $U_{n_i}$'s are replaced by $V_{n_i}$'s although we have not seen this being explicitly done in the literature. 
Here, we prove the following theorem.

\begin{theorem}\label{theorem1}
The equation  \eqref{eq:1} with $t=1$ implies $n_1\le 3\times 10^5$. When $\alpha, \beta $ are real, then $n_1\le 210$. Further, if $s=\pm 1$, then $n_1\leq 150$. The same results hold if in \eqref{eq:1} with $t=1$ we replace $U_{n_1}$ by $V_{n_1}$.   
\end{theorem}

We leave it as a challenge to the reader to prove (and find a value of) that there exists $n_0$ which is absolute such that the largest solution of \eqref{eq:1} with $1\le n_1<n_2<\cdots<n_t$ (where $t$ is also a variable) satisfies $n_t<n_0$. Throughout the proof, we use $\omega(n), P(n), \mu(n), \varphi(n)$ with the regular meaning as being the number of distinct prime factors of $n$, the largest prime factor of $n$, 
the M\"obius function of $n$ and the Euler function of $n$, respectively. 

\section{Proof of Theorem \label{th1}}

We first treat the case of the sequence $\{U_n\}_{n\ge 0}$. At the end we indicate the slight change needed to cover the case of the sequence $\{V_n\}_{n\ge 0}$. We assume without loss of generality that $|\alpha|\ge |\beta|$. We may also assume that $n\ge 150$ is such that 
$U_n=\pm m_1! m_2!\cdots m_k!$ where $k\ge 1$ and $1<m_1\leq \cdots\leq m_k$. Since $n\ge 150$, 
$U_n$ has a primitive prime factor (see \cite{BHV}), which is a prime congruent to $\pm 1\pmod n$. This prime must divide $m_k!$, so $m_k\ge rn-1$ with $r=1$ if $n$ is even and $r=2$ if $n$ is odd. Thus, using $m!\ge {\sqrt{2\pi}} (m/e)^m>2(m/e)^m$, we have
$$
2|\alpha|^{n}\ge \left|\frac{\alpha^n-\beta^n}{\alpha-\beta}\right|=|U_n|\ge m_k!\ge (n-1)!\ge 2\left(\frac{rn-1}{e}\right)^{rn-1},
$$
so that 
\begin{align*}
\log |\alpha| \geq &\left(r-\frac{1}{n}\right)(\log (rn-1)-1)\\
=&\left(r-\frac{1}{4}\right)\log n+\left(\frac{1}{4}-\frac{1}{n}\right)\log n
+\left(r-\frac{1}{n}\right)\log\left(r-\frac{1}{n}\right)-\left(r-\frac{1}{n}\right)\\
\geq &\left(r-\frac{1}{4}\right)\log n+\left(\frac{1}{4}-\frac{1}{100}\right)\log 100
+\left(r-\frac{1}{100}\right)\log\left(r-\frac{1}{100}\right)\\
& -\left(r-\frac{1}{100}\right)
\end{align*}
since $n\geq 150>100$.  Taking $r\in \{1, 2\}$, we see that 
\begin{equation}
\label{eq:n0}
\log |\alpha|>\begin{cases}
\frac{3}{4}\log n \ & {\rm if} \ n \ {\rm is \ even};\\
\frac{7}{4}\log n \ & {\rm if} \ n \ {\rm is \ odd}.
\end{cases}
\end{equation}
In particular, $\log |\alpha|>\frac{3}{4}\log n$ for all $n$.  We now look at the Primitive Part of $U_n$. This is the part of $U_n$ built up only with primitive prime divisors $p$ which are those primes that do not divide $U_m$ for any $1\le m\le n-1$, and also do not divide $\Delta=r^2+4s$. Since $n>30$, these primes exist and they are all congruent to $\pm 1\pmod n$. Further, it is well-known (see, for example, Theorem 2.4 in \cite{BHV}), that 
$$
\prod_{\substack{p^{\alpha_p}\| u_n\\ p~{\text{\rm primitive}}}} p^{\alpha_p}=\frac{\Phi_n(\alpha,\beta)}{\delta},
$$
where 
$$
\Phi_n(\alpha,\beta)=\prod_{\substack{1\le k\le n\\ (k,n)=1}} (\alpha-e^{2\pi i k/n} \beta),
$$
is the specialisation of the homogenization $\Phi_n(X,Y)$ of the $n$th cyclotomic polynomial $\Phi_n(X)$ in the pair $(\alpha,\beta)$, while $\delta\in \{2,3,P(n)\}$. Here, $P(n)$ is the largest prime factor of $n$ as stated before.   Thus, in particular, 
\begin{equation}
\label{eq:2}
M_n:=\prod_{\substack{p^{\alpha_p}\| u_n\\ p\equiv \pm 1\pmod n}} p^{\alpha_p}\ge \prod_{\substack{p^{\alpha_p}\| u_n\\ p~{\text{\rm primitive}}}} p^{\alpha_p}\ge \frac{|\Phi_n(\alpha,\beta)|}{n}.
\end{equation}
It is well-known that
\begin{eqnarray}
\label{eq:3}
\Phi_n(\alpha,\beta) & = & \prod_{m|n} (\alpha^{n/m}-\beta^{n/m})^{\mu(m)};\\
|\alpha^m-\beta^m| & \le 2|\alpha|^m.\nonumber
\end{eqnarray}
If in addition $\alpha$ and $\beta$ are real, then the inequality 
\begin{equation}
\label{eq:4}
|\alpha^m-\beta^m|\ge |\alpha|^{m-1}
\end{equation}
holds. In this case, it is well-known and it follows easily from  \eqref{eq:3} and \eqref{eq:4} and $\sum_{m|n}\mu(m)n/m=\varphi(n)$ 
that
\begin{equation}
\label{eq:5}
\log |\Phi_n(\alpha,\beta)|\ge \varphi(n)\log |\alpha|-2^{\omega(n)-1}\left(\log 2+\log|\alpha|\right).
\end{equation}
When $s=\pm 1$, we can do much better. Namely in this case $\beta=\pm \alpha^{-1}$ and $|\alpha|\ge (1+{\sqrt{5}})/2$. Hence, from \eqref{eq:3}, one gets easily that
$$
\Phi_n(\alpha,\beta)\ge |\alpha|^{\phi(n)} \prod_{d\ge 1} \left(1-\frac{1}{\alpha^{2d}}\right)\left(1+\frac{1}{\alpha^{2d}}\right)^{-1}>|\alpha|^{\phi(n)}\times 0.278293,
$$
so
\begin{equation}
\label{eq:55}
\log |\Phi_n(\alpha,\beta)|\ge \varphi(n)\log |\alpha|-1.28.
\end{equation}
When $\alpha$ and $\beta$ are complex conjugates, a lower-bound on the left--hand side of \eqref{eq:4} can be obtained using a linear 
form in two complex logarithms \'a la Baker. This was worked out in \cite{Vou} 
(see Lemma $5 (ii)$ and Theorem $2 (ii)$ in \cite{Vou}) and given for $m\geq 3$ by both 
\begin{align}
\label{a-b1}
\log |\alpha^m-\beta^m|\ge m\log |\alpha|-\left( \frac{m}{\gcd(m, 2)}+\frac{\log 2}{4}+0.02\right)\log |\alpha| 
\end{align}
\begin{align}\label{a-b2}
{\rm and} \ &\log |\alpha^m-\beta^m|\ge m\log |\alpha|-73\log |\alpha| \left(\log \frac{m}{\gcd(m, 2)}\right)^2.
\end{align}
For $m\in \{1, 2\}$, we have $|\alpha^m-\beta^m|\ge (m-1)\log |\alpha|$. 
As also remarked in \cite{Vou}, the inequality \eqref{a-b1}  is better when $m\leq 5358$. Using \eqref{a-b2}, we obtain 
(as in the expression between displays (9) and (10) on \cite[p416]{Vou})  that 
\begin{equation}
\label{eq:61}
\log |\Phi_n(\alpha,\beta)|\ge (\varphi(n)-1)\log |\alpha|-2^{\omega(n)-1}\log 2-73\log |\alpha| f(n),
\end{equation}
where
\begin{equation}
\label{eq:fn}
f(n)\le \sum_{\substack{m\mid n\\ \mu(m)=1}} \log(n/m)^2.
\end{equation}
Since $f(n)\le 2^{\omega(n)-1} (\log n)^2$, we get
\begin{equation}
\label{eq:6}
\log |\Phi_n(\alpha,\beta)|\ge (\varphi(n)-1)\log |\alpha|-2^{\omega(n)-1}(\log 2+73\log |\alpha| (\log n)^2).
\end{equation}
In particular, using \eqref{eq:2} as well as \eqref{eq:5} and \eqref{eq:6}, we get 
\begin{equation}
\label{eq:7}
\log M_n\ge (\varphi(n)-1)\log |\alpha|-\log n-2^{\omega(n)-1}\left(\log 2+73\log|\alpha| (\log n)^2\right).
\end{equation}
We compare the above bound with an upper bound for $\log M_n$ which we obtain in the following way. We use sieves to get an 
upper bound on $\log M_n$ in terms of 
\begin{equation}
\label{eq:8}
\sum_{m_i\ge n-1} m_i(\log m_i-1).
\end{equation}
Then we get an upper bound on \eqref{eq:8} in terms of $n$ and $|\alpha|$. Finally, we match those two and we get an inequality relating $n$ and $|\alpha|$ which we exploit.

Let's get to work. Using again the fact that $m!\ge \sqrt{2\pi m}(m/e)^m>2(m/e)^m$, we get
$$
2|\alpha|^n\ge |u_n|\ge \prod_{i} m_i!\ge \prod_{m_i\ge n-1} 2(m_i/e)^{m_i},
$$
and taking logarithms we get
\begin{align}\label{nloga}
n\log |\alpha|\ge \sum_{m_i\ge n-1} m_i(\log m_i-1).
\end{align}
We now get an upper bound on $\log M_n$ in terms of the sum shown at \eqref{nloga}. For that, note that for any prime $p\equiv \pm 1\pmod n$, we have
\begin{eqnarray*}
\alpha_p & = & \nu_p(M_n)\le \nu_p(u_n)=\nu_p\left(\prod_{i=1}^k m_i!\right)\le  \sum_{m_i\ge n-1} \frac{m_i-1}{p-1}
\end{eqnarray*}
by using the fact that 
$$
\nu_p(k!)=\frac{k-\sigma_p(k)}{p-1}\leq \frac{k-1}{p-1},
$$ 
where $\sigma_p(k)$ is the sum of digits of $k$ in base 
$p$. Hence,
\begin{equation}
\label{eq:10}
\log M_n\le \sum_{p\equiv \pm 1\pmod n} \alpha_p \log p\le 
\sum_{m_i\ge n-1} (m_i-1)\sum_{\substack{p\le m_i \\ p\equiv \pm 1\pmod n}}\frac{\log p}{p-1}.
\end{equation}
It remains to evaluate the inner sums on the right above. We use a variation of an argument from \cite{Lu2}. That we split into two parts. When $p<3n$, we 
have $\log p<\log(3n)$ and
\begin{align*}
\sum_{\substack{p<3n\\ p\equiv \pm 1\pmod n}} \frac{\log p}{p-1}\leq &
\begin{cases} 
\left(\frac{1}{n-2}+\frac{1}{n}+\frac{1}{2n-2}+\frac{1}{2n}+\frac{1}{3n-2}\right)\log (3n) & \ {\rm if } \ n \ {\rm is \ even};\\
\left(\frac{1}{2n-2}+\frac{1}{2n}\right)\log (3n) & \ {\rm if } \  n \ {\rm is \ odd};
\end{cases}\\
\leq &\begin{cases} 
\frac{10.1\log (3n)}{3n} & \ {\rm if } \ n \ {\rm is \ even};\\
\frac{3.1\log (3n)}{3n} & \ {\rm if } \  n \ {\rm is \ odd};
\end{cases}
\end{align*}
since $n\geq 150$.  For $p>3n$, we have
\begin{align*}
\sum_{\substack{p\equiv \pm 1 \pmod n\\ p>3n}}\frac{\log p}{p-1}\leq &
\sum_{\substack{p\equiv \pm 1\pmod n\\ p>3n}}\frac{\log p}{p}+
\sum_{t\ge 3n+1}\left(\frac{\log (t-1)}{t-1}-\frac{\log t}{t}\right)\\
\leq & \sum_{\substack{p\equiv \pm 1\pmod n\\ p>3n}} \frac{\log p}{p}+\frac{\log(3n)}{3n}.
\end{align*}
Since $\varphi(n)\le n/2$  holds for $n$ even, we obtain $\frac{11.1\log (3n)}{3n}\leq \frac{11.1\log (3n)}{6\varphi(n)}$ 
for  $n$ even. The inequality $\frac{4.1\log (3n)}{3n}\leq \frac{11.1\log (3n)}{6\varphi(n)}$ also holds for $n$ odd. Therefore, we get that 
\begin{align}\label{logp}
\sum_{\substack{p\equiv \pm 1 \pmod n\\ p\leq m}}\frac{\log p}{p-1}\leq 
\sum_{\substack{p\equiv \pm 1\pmod n\\ 3n<p\leq m_i}} \frac{\log p}{p}+\frac{11.1\log(3n)}{6\varphi(n)}.
\end{align}
We use the estimate 
\begin{align}\label{pi-x-a}
\pi(x;n,a)\le \frac{2x}{\varphi(n)\log (x/n)} 
\end{align}
which holds for both $a\in \{\pm 1\}$ and when $x>n$, where $\pi(x;n,a)$ stands for the number of primes $p\le x$ satisfying 
$p\equiv a\pmod n$.  For simplicity, we put $\pi_1(x):=\pi(x; n,1)$ and $\pi_{-1}(x):=\pi(x; n,-1)$.  By Abel's summation 
formula, we have 
\begin{align*}
\sum_{\substack{3n<p\le m_i\\ p\equiv \pm 1\pmod n}} \frac{\log p}{p} 
\le & \left(\pi_1(m_i)+\pi_{-1}(m_i)\right)\frac{\log m_i}{m_i}-\left(\pi_1(3n)+\pi_{-1}(3n)\right)\frac{\log(3n)}{3n}\\
& - \int_{3n}^{m_1} \left(\pi_1(t)+\pi_{-1}(t)\right)\frac{1-\log t}{t^2} dt\\
\le &\frac{4}{\varphi(n)}\frac{m_i}{\log(m_i/n)}\frac{\log m_i}{m_i}
+ \frac{4}{\varphi(n)}\int_{3n}^{m_i} \frac{\log(t/n)+\log n-1}{t\log(t/n)}dt\\
\le &\frac{4}{\varphi(n)}\frac{\log m_i}{\log (m_i/n)}
+ \frac{4}{\varphi(n)}\int_{3n}^{m_i} \left(\frac{1}{t}+\frac{\log n-1}{t\log(t/n)}\right)dt\\
< &\frac{4}{\varphi(n)}\left(
\frac{\log m_i}{\log (m_i/n)}+\log m_i-\log 3n+(\log n-1)(\log \log (m_i/n))\right)
\end{align*}
We put $m_i=n^{1+c}$ with $c\geq \log 3/\log n$ since $m_i\geq 3n$.  Hence, we have from \eqref{logp} that 
$\sum_{\substack{p\le m_i\\ p\equiv \pm 1\pmod n}} \frac{\log p}{p}$ is 
\begin{align*}
\leq & \frac{4}{\varphi(n)}\left(\frac{c+1}{c}+(c+1)\log n-\log 3n+(\log n-1)(\log \log n+\log c)+\frac{11.1\log 3n}{24}\right)\\
\leq & \frac{4((c+1)\log n-1)}{\varphi(n)}\left( 1+\frac{(\log n-1)}{(c+1)\log n-1}(\log \log n+\log c))-
\frac{\frac{12.9}{24}\log 3n-2-\frac{1}{c}}{(c+1)\log n-1}\right)\\
\leq & \frac{4(\log m_i-1)}{\varphi(n)}\left( 1+\frac{\log \log n}{c+1}+\frac{\log c}{c+1}-
\frac{\frac{12.9}{24}\log 3n-2-\frac{1}{c}}{(c+1)\log n-1} \right).
\end{align*}
If $c\geq 1$; that is, if $m_i\ge n^2$, then the expression inside the bracket is at most $1+(\log\log n)/2$. If
$c<1$ but $(\log 3n)/2-2-1/c\geq 0$, the expression inside the bracket is at most $1+\log\log n$. Assume now that
$(\log 3n)/2<2+1/c$. Since $n^c\geq 3$, we have $1/c\leq \frac{\log n}{\log 3}$ and therefore
\begin{eqnarray*}
&& -\frac{\frac{12.9}{24}\log 3n-2-\frac{1}{c}}{(c+1)\log n-1} \leq \frac{\frac{\log n}{\log 3}+2-\frac{12.9\log 3n}{24}}{\log n+2}
\\
& \leq &  \frac{0.38\log n+1.5}{\log n+2}\leq .38+\frac{.74}{\log n+2}\leq .5
\end{eqnarray*}
Again, from $1/c> (\log 3n)/2-2\geq \frac{\log n}{3}$, we have $\log c<-\log \log n+\log 3$ and therefore 
\begin{align*}
1+\frac{\log \log n}{c+1}+\frac{\log c}{c+1}-\frac{\frac{12.9}{24}\log 3n-2-\frac{1}{c}}{(c+1)\log n-1} 
\leq 1+\frac{\log 3}{c+1}+.5\\
< 1.6<1+\log \log n 
\end{align*}
since $n\geq 150$.  Therefore, we have 
\begin{align}\label{ublogp}
\sum_{\substack{p\le m_i\\ p\equiv \pm 1\pmod n}} \frac{\log p}{p} \leq \begin{cases}
\frac{4(\log m_i-1)}{\varphi(n)}\left(1+\log \log n\right) \  & {\rm if } \ m_i<n^2;\\
\frac{4(\log m_i-1)}{\varphi(n)}\left(1+\frac{\log \log n}{2}\right) \  &  {\rm if } \ m_i\geq n^2.
\end{cases}
\end{align}
Putting this this into \eqref{eq:10}, we get
$$
\log M_n\le \left(\frac{4(1+\log\log n)}{\varphi(n)}\right)\left(\sum_{m_i\ge n-1} (m_i-1)(\log m_i-1)\right),
$$
which combined with \eqref{nloga} gives
\begin{equation}
\label{eq:14}
\log M_n\le  \left(\frac{4(1+\log\log n)}{\varphi(n)}\right) n\log |\alpha|.
\end{equation}
Combining \eqref{eq:14} with \eqref{eq:7}, we get
\begin{eqnarray*}
 (\varphi(n)-1)\log |\alpha| & - & \log n-2^{\omega(n)-1}\left(\log 2+73\log|\alpha| (\log n)^2\right)\\
 & \le & \left(\frac{4(1+\log\log n)}{\varphi(n)}\right) n\log |\alpha|.
 \end{eqnarray*}
 This is equivalent to
 \begin{eqnarray*}
\frac{\varphi(n)}{n} \left\{\varphi(n)-1-\frac{\log n}{\log |\alpha|}-\frac{2^{\omega(n)} \log 2}{2\log |\alpha|}-
 73\times 2^{\omega(n)-1}  (\log n)^2\right\}
  \le 4+4\log\log n.
 \end{eqnarray*}
 Using $\log |\alpha|\ge 0.5\log n$ from \eqref{eq:n0} as well as effective estimates from prime number theory given by 
 $$
 \varphi(n)>\frac{n}{e^{\gamma}\log\log n+2.50637/\log \log n}\qquad {\text{\rm and}}\qquad \omega(n)<\frac{1.3841\log n}{\log\log n},
 $$
 (see Th\'eor\`eme 11 of \cite{Ro} and Theorem 15 of \cite{Ros}) with $\gamma=0.57721\ldots<0.5722$,  we get $n<18\times 10^6$. 
 We obtain $\omega(n)\le 8$, 
 $$
 \varphi(n)\ge n\prod_{k=1}^8 \left(1-\frac{1}{p_k}\right),
 $$
 where $p_k$ is the $k-$th prime. From Voutier  \cite[Lemma 7]{Vou}, we get an improvement  of the trivial inequality 
 $f(n)\le 128 (\log n)^2$ to 
 $$
 f(n)\le 128 (\log n)^2-1886 (\log n)+7913, 
 $$
 and substituting this into \eqref{eq:61}  and redoing the above calculation we get $n<3.9\times 10^6$. Then $\omega(n)\le 7$. Using
 $$
 \varphi(n)\ge n\prod_{k=1}^7 \left(1-\frac{1}{p_k}\right),
 $$
 and the fact that $f(n)\le 64 (\log n)^2-775 \log n+2718$ when $\omega(n)\leq 7$ proved also in \cite[Lemma 7]{Vou}, we get 
 $n<1.852\times 10^6$. We improve this bound further.  First we observe that $n$ is even if $\omega(n)=7$ else 
 $n\geq \prod^8_{i=2}p_i>4\cdot 10^6$.  We first prove the following lemma which we need for reducing the bound further. 
 This ideas can be exploited further by those who would like to reduce the bound further.  
\begin{lemma}
For $n$ odd, we have 
\begin{align}
\label{w7}
\log M_n\ge &(\varphi(n)-1)\log |\alpha|-(1+\frac{2^{\omega}}{4\omega})\log n-2^{\omega-2}\log 2-g_{\omega}\log  |\alpha|,
\end{align}
where $\omega=\omega(n)$ and 
\begin{align}
\label{gw}
g_{\omega}=\begin{cases}
73(11\log^2 n-87.5\log n+194.1)+0.0027n+3.1 & \ {\rm if} \ \omega=6;\\
73(7\log^2 n-49.1\log n+101.6)+\frac{n}{1155}+0.2 & \ {\rm if} \ \omega=5;\\
73(4\log^2 n-22.6\log n+43.1) & \ {\rm if} \ \omega=4;\\
73(2\log^2 n-6.8\log n+11.6) & \ {\rm if} \ \omega=3;\\
73\log^2 n & \ {\rm if} \ \omega\leq 2.
\end{cases}
\end{align}

For $n$ even, we have 
\begin{align} \label{w8}
\log M_n\ge &(\varphi(n)-1)\log |\alpha|-\log n-2^{\omega-2}\log 2-h_{\omega}\log  |\alpha|,
\end{align}
where $\omega=\omega(n)$ and 
\begin{align}
\label{hw}
h_{\omega}=\begin{cases}
73(16 \log^2 \frac{n}{2}-139\log \frac{n}{2}+327)+0.0032n+3.1 & \ {\rm if} \ \omega=7;\\
73(11 \log^2 \frac{n}{2}-87.5\log \frac{n}{2}+194.1)+0.002n+0.97 & \ {\rm if} \ \omega=6;\\
73(7 \log^2 \frac{n}{2}-49.1\log \frac{n}{2}+101.6)+0.0005n+0.2 & \ {\rm if} \ \omega=5;\\
73(4 \log^2 \frac{n}{2}-22.6\log \frac{n}{2}+43.1) & \ {\rm if} \ \omega=4;\\
73(2\log^2 \frac{n}{2}-6.8\log \frac{n}{2}+11.6) & \ {\rm if} \ \omega=3;\\
73\log^2 \frac{n}{2} & \ {\rm if} \ \omega\leq 2.
\end{cases}
\end{align}

\end{lemma}

\begin{proof} 
We write $n=ab$ where $a$ is the radical of $n$, i.e., the product of distinct primes dividing $n$.  Every divisor $d|n$ with 
$\mu(d)=\pm 1$ is also a divisor of $a$. We have from \eqref{eq:3} that
\begin{align*}
\log |\Phi(\alpha, \beta)|=\sum_{d|a}\mu(d)\log |\alpha^{n/d}-\beta^{n/d}|=\varphi(n)\log |\alpha|+\sum_{d|a}\mu(d) \log |1-x^{n/d}| 
\end{align*}
where $x=\frac{\beta}{\alpha}$.  We will estimate  $\sum_{d|a}\mu(d) \log |1-x^{n/d}|$ to prove \eqref{w7}.  When 
$\omega(n)=1$ with $n=p^r$, we have $\log  |1-x^n|-\log |1-x^{n/p}|\geq \log  |1-x^{n}|-\log 2$ and the assertion 
follows from \eqref{a-b2}. Hence, we consider $\omega(n)\geq 2$. 

Let $q$ be the least prime divisor of $n$. Write $a=qa_0$ so that 
$q\nmid a_0$. Then every divisor $d|a_0$ gives two distinct divisors $d$ and $qd$ of $a_0$. 
Let $d|a_0$ with $\mu(d)=-1$. This gives $\mu(qd)=1$ and we have 
\begin{align}\label{l2}
\mu(d) \log |1-x^{n/d}|+\mu(qd) \log |1-x^{n/qd}|=\log \frac{|1-x^{n/qd}|}{|1-x^{n/d}|}\geq -\log q.
\end{align}
 For $d|a_0$ with  $\mu(d)=1$, we have from $\mu(qd)=-1$ that 
 \begin{align*}
\mu(d) \log |1-x^{n/d}|+\mu(qd) \log |1-x^{n/qd}|\geq  \log |1-x^{n/d}|-\log 2.
\end{align*}
If $\mu(d)=1$ and $q=2$, then
 \begin{align}\label{l3}
\mu(d)\log |1-x^{n/d}|+\mu(2d) \log |1-x^{n/dq}|= \log |1+x^{n/2d}|.
\end{align}
Let $n$ be odd. Then $\omega=\omega(n)\leq 6$ and $q>2$.  We have 
\begin{align*}
\sum_{d|a}\mu(d) \log |1-x^{n/d}|\geq \sum_{\underset{\mu(d)=1}{d|a_0}} \log |1-x^{n/d}|-2^{\omega-2}\log(2q)
\end{align*}
since $\omega(a_0)=\omega -1$. Observe that $d|a_0$ are squarefree. 
Now we use \eqref{a-b1} for $\omega(d)=4$ and \eqref{a-b2} for $d=1$ or $\omega(d)=2$ along with 
$\log q\leq \frac{\log n}{\omega(n)}$ to obtain \eqref{w7}.  For  $k\in \{2, 4\}$, there are  
$\binom{\omega-1}{k}$ number of $d$'s with $d|a_0$ and $\omega(d)=k$. Let 
$d_{k, 1}<d_{k, 2}<\cdots $ be the sequence of odd squarefree numbers with each $\omega(d_{ki})=k$. 
We obtain 
\begin{align*}
 \sum_{\underset{\mu(d)=1}{d|a_0}} \log |1-x^{n/d}|\geq &-73\log \alpha \left((\log n)^2+\sum^{\binom{\omega-1}{2}}_{i=1}
 \log^2\left(\frac{n}{d_{2, i}}\right)\right)\\
 &-\log |\alpha| \left(n\sum^{\binom{\omega-1}{4}}_{i=1}\frac{1}{d_{4, i}}+\binom{\omega-1}{4}
 \left(\frac{\log 2}{4}+0.02\right)\right)\\
 \geq &-(\log |\alpha|) g_{\omega}
\end{align*}
for each $\omega(n) \leq 6$ and by expanding $\log^2 (n/d)=\log ^2n-2(\log d)(\log n)+\log ^2d$.  

Let $d$ be even and hence $q=2$. We obtain from \eqref{l2} and \eqref{l3} that 
\begin{align*}
\sum_{d|a}\mu(d) \log |1-x^{n/d}|\geq \sum_{\underset{\mu(d)=1}{d|a_0}} \log |1+x^{n/2d}|-2^{\omega-2}\log 2. 
\end{align*}
The right hand side of the inequalities \eqref{a-b1} and \eqref{a-b2} are also lower bounds for 
$\alpha^m+\beta^m$ for $m\geq 3$ (see proof of \cite[Lemma 5]{Vou}). Observe that 
$d|a_0$ are odd squarefree. As in the $n$ odd case, we use \eqref{a-b1} for $\omega(d)=4, 6$ and 
\eqref{a-b2} for $d=1$ or $\omega(d)=2$ to obtain \eqref{w8}. We obtain 
\begin{align*}
 \sum_{\underset{\mu(d)=1}{d|a_0}} \log |1+x^{n/2d}|\geq &-73\log \alpha \left((\log n/2)^2+\sum^{\binom{\omega-1}{2}}_{i=1}
 \log^2\left(\frac{n}{2d_{2, i}}\right)\right)\\
 &-\log |\alpha| \left(n\sum^{\binom{\omega-1}{4}}_{i=1}\frac{1}{2d_{4, i}}+\binom{\omega-1}{4}
\left(\frac{\log 2}{4}+0.02\right)\right)\\
 &-\log |\alpha| \left(n\sum^{\binom{\omega-1}{6}}_{i=1}\frac{1}{2d_{6, i}}+\binom{\omega-1}{6}
 \left(\frac{\log 2}{4}+0.02\right)\right)\\
 \geq &-(\log |\alpha|) h_{\omega}.
\end{align*}
Hence the assertion.  
\end{proof}
Now we combine the above lower bound for $\log M_n$  with the upper bound given by \eqref{eq:14} and use 
$\varphi(n)\ge n\prod_{k=1}^{\omega (n)} \left(1-\frac{1}{p_k}\right)$ if $n$ is even and  
$\varphi(n)\ge n\prod_{k=2}^{\omega (n)+1} \left(1-\frac{1}{p_k}\right)$ if $n$ is odd. We obtain 
$n\leq 500000$ implying $\omega(n)\leq 6$. Further, we get $n<270000$ if $n$ is even and 
$n<150000$ if $n$ is odd. This implies the first assertion of the theorem.

Now in case $\alpha, \beta $ are real, we use \eqref{eq:5} instead of \eqref{eq:7} with $\omega (n)\in \{3, 4, 5, 6, 7\}$ and 
using $\varphi(n)\ge n\prod_{k=1}^{\omega (n)} \left(1-\frac{1}{p_k}\right)$, we obtain $n\leq 167, 252, 1000$ according 
to whether $\omega(n)=3, 4$ or $\omega(n)>4$, respectively. Thus, $\omega(n)\leq 4$ and further 
$n\le 2\cdot 3\cdot 5\cdot 7=210$. 

When $s=\pm 1$, we use \eqref{eq:55} instead and check that there is no value $n\in [151,210]$ for which 
the resulting inequality holds. 

Finally,  we replace $\{U_n\}_{n\ge 0}$ by $\{V_n\}_{n\ge 0}$ in \eqref{eq:1}. The first ingredient of the problem was the upper bound $|U_n|\le 2|\alpha|^n$ which holds when $U_n$ is 
replaced by $V_n$ as well. As for the ``primitive part", since $V_n=U_{2n}/U_n$, it follows that in fact we have the better inequality that the primitive part of $V_n$ is at least as large as $\log |\Phi_{2n}(\alpha,\beta)|/\delta$. 
In addition, the primitive prime factors of $V_n$ are congruent to $\pm 1$ modulo $2n$. The above arguments now imply immediately that the same conclusion holds for this case and in fact that $2n\le 3\times 10^5,~210,~150$ 
for the general case, real $\alpha,~\beta$ case, and $s=\pm 1$ case, respectively. 

\section*{Acknowledgements}

Part of this work was done when the first author visited the School of Maths of Wits University in December 2018. 
 He thanks this Institution for hospitality and CoEMaSS Grant RTNUM18 for financial support.   The second  author 
 was supported in part by NRF (South Africa) Grant CPRR160325161141 and by CGA (Czech Republic) Grant 17-02804S.

\end{document}